\documentclass[11pt]{article}
\usepackage{geometry}
\usepackage{hyperref}

\usepackage{amsmath}
\usepackage{amssymb}
\usepackage{amsthm}
\usepackage{mathtools}
\usepackage{physics}
\usepackage{relsize}
\usepackage{float}
\usepackage{graphicx}
\usepackage{subcaption}
\usepackage[justification=centering]{caption}
\usepackage{bigints}
\usepackage{import}
\usepackage{titlepic}
\usepackage[backend=biber,style=ieee]{biblatex}
\mathtoolsset{showonlyrefs}
\usepackage{pdfpages}
\theoremstyle{plain}
\newtheorem{theorem}{Theorem}[section]
\newtheorem{lemma}[theorem]{Lemma}
\newtheorem{proposition}[theorem]{Proposition}
\newtheorem{corollary}[theorem]{Corollary}

\newtheorem{definition}[theorem]{Definition}
\theoremstyle{remark}
\newtheorem{remark}[theorem]{Remark}
\numberwithin{equation}{section}
\usepackage[symbol]{footmisc}

\title{Beyond periodic revivals for linear dispersive PDEs}
\author{
Lyonell Boulton, George Farmakis and  Beatrice Pelloni \\
Heriot-Watt University \& Maxwell Institute  for the Mathematical Sciences\\ Edinburgh, Scotland
}

\begin{document}
\maketitle
\begin{abstract}
We study the phenomenon of revivals for the  linear Schr{\"odinger} and Airy equations over a finite interval, by considering several types of non-periodic boundary conditions.  In contrast with the case of the linear Schr\"odinger equation examined recently (which we develop further), we prove that, remarkably,  the Airy equation does not generally exhibit revivals even for boundary conditions very close to periodic. We also describe a new, weaker form of revival phenomena, present in the case of certain Robin-type boundary conditions for the linear Schr\"{o}dinger equation. In this weak revival, the dichotomy between the behaviour of the solution at rational and irrational times persists, but in contrast with the classical periodic case, the solution is not given by a finite superposition of copies of the initial condition. 
\end{abstract}

\section{Introduction}\label{Introduction}
The phenomenon of {\em revivals} in linear dispersive periodic problems, also called  in the literature {\em Talbot effect} or  {\em dispersive quantisation}, has been well-studied and is by now well-understood. It was discovered first experimentally in optics, then rediscovered several times by theoretical and experimental  investigations.  While the term has been used systematically and consistently by many authors, there is no consensus on a rigorous definition.  Several have described the phenomenon by stating that a given periodic time-dependent boundary value problem exhibits {\em revival at rational times} if the solution evaluated at a certain dense subset of times, see \eqref{Rational Time} below, is given by  finite superpositions of translated copies of the initial profile.  We will call this the {\em periodic revival property}. In particular, when the initial condition has a jump discontinuity at time zero, these discontinuities are propagated and remain present in the solution at each rational time. 

This  behaviour  at rational times should be contrasted with the behaviour at generic time, when the solution is known to be continuous as soon as the initial condition is of bounded variation.  Hence generically, while the dispersive propagation has a smoothing effect on any initial discontinuity,  this smoothing does not occur at rational times. Moreover, at generic times and for appropriate initial data,  the solution while continuous is nowhere differentiable. In fact its graph has a fractal dimension greater than 1 \cite{berry1996integer, rodnianski2000fractal}. There is therefore a dichotomy between generic times and the measure-zero set of rational times, as suggested also by the provocative title of \cite{kapitanski1999does}. 

In this paper we  examine the role of boundary conditions in supporting some form of  revival phenomenon. In order to illustrate the range of revival behaviour concretely, we focus on two specific linear PDEs of particular significance both from the mathematical point of view and in terms of applications. Namely, we will consider the linear Schr\"odinger equation with zero potential 
\begin{align}
    & iu_t(x,t)+u_{xx}(x,t)=0,
\label{linS} \tag{LS}
\end{align}
and the  Airy equation, also know as Stokes problem,
\begin{align}
&u_t(x,t)  -u_{xxx}(x,t)=0.
\label{airy} \tag{AI}
\end{align}
Both these PDEs will be posed on the interval $[0,2\pi]$ and we set specific boundary conditions either of pseudo-periodic or of Robin type.  These represent two essentially  different types  of boundary conditions. Indeed, in the pseudo-periodic case the boundary conditions couple the ends of the interval, just as in the periodic case, while in the Robin\ case, the boundary conditions are uncoupled. The  type of  revival property that we observe in the two cases strongly reflects this difference.   

Let $n$ denotes the order of spatial derivative in the PDE, hence  $n=2$ for \eqref{linS} and $n=3$ for \eqref{airy}. In the first part of the paper, following the work of \cite{olver2018revivals}, we will consider specific types of {\em pseudo-periodic} boundary conditions of the form\begin{equation}
\beta_k\partial_x^k u(0,t)=u(2\pi,t),\quad \beta_k\in\mathbb{C},\;\;k=0,1,...,n-1 .
\label{ppbc} \tag{PP}
\end{equation}
Of particular interest will be the case when all $\beta_k$  are equal, that is {\em quasi-periodic} boundary conditions of the form
\begin{equation}
\beta \partial_x^ku(0,t)=u(2\pi,t),\quad \beta\in\mathbb C,\;\;k=0,1,...,n-1 .
\label{qpbc} \tag{QP}
\end{equation}
In the second part of the paper, we will consider Schr\"{o}dinger's equation \eqref{linS} with the specific Robin boundary conditions given by
\begin{equation}
b u(x_{0},t) = (1-b) \partial_{x}u(x_{0},t),\  x_{0} = 0, \ \pi,\quad b \in [0,1].
\label{rbc} \tag{R}
\end{equation}
The case $b=0$ corresponds to Neumann and $b=1$ to Dirichlet boundary conditions. For these  special cases, the solution of the boundary value problem is obtained by even or odd extensions from the solution of an associated periodic problem. However, for $0<b<1$ the boundary value problem behaves very differently from a periodic one.

It is well established that the  periodic problem for any linear dispersive equations exhibits periodic revival (see Theorem~\ref{Linear Evolution Problems}  below). Moreover, subject to consistency conditions on the coefficients $\beta_k$, in \cite{olver2018revivals} it was shown that the periodic revival property holds in general for the equation \eqref{linS}-\eqref{ppbc}. Below, we give a new proof of the latter. Our arguments elucidate the mathematical reason for the persistence of the periodic revival property for the linear Schr{\"o}dinger equation \eqref{linS}, when subject to the fairly general class of boundary conditions \eqref{ppbc}. In particular, we show that all pseudo-periodic boundary conditions can be solved in terms of certain associated {\em periodic} problems. This is the content of Proposition \ref{Non-Self-Adjoint Correspondence Theorem} which also enables us to deduce from existing results for the periodic case that, at irrational times, any initial discontinuity is smoothed out. To be precise, even when the initial profile has jump discontinuities, the solution at irrational times becomes a continuous (though nowhere differentiable) function of the space variable. 

The spectral properties of pseudo-periodic and other non-periodic boundary value problems for \eqref{airy} were first examined in \cite{pelloni2005spectral}, where an explicit general formula for the solution was given. Below we show that, in stark contrast with \eqref{linS}, the quasi-periodic  Airy equation, \eqref{airy}-\eqref{qpbc}, in general does not  exhibit any form of revival at rational times. Indeed, the periodic revival property holds in this case {\em only for  values of the quasi-periodicity parameter}, $\beta=e^{2\pi i \theta}$ such that  $\theta\in\mathbb Q$. Remarkably, the latter defies the na\"ive expectation that the revival property carries onto the case of higher order PDEs, when the boundary conditions support it for the second order case. It also suggests that the general pseudo-periodic case for third-order PDEs, generically,  will not  exhibit revivals. 

In Section~\ref{Airy's Quasi-Periodic Problem} we prove the following result.
\begin{theorem}
\label{Airy Correspondence Theorem}
Fix $\theta \in [0,1)$ and consider Airy's equation \eqref{airy} with initial condition $u_{0}\in L^{2}(0,2\pi)$ and quasi-periodic boundary conditions \eqref{qpbc} where $\beta = e^{i2\pi\theta}$. Let $p$ and $q$ be co-prime and let 
	\begin{equation}
		\label{Airy correspondence IC}
	 v^{(p,q)}_{0}(x) = \mathcal{R}_{3}(p,q)\left[u_{0}(x) e^{-i\theta x}\right],
	\end{equation}
	where $\mathcal{R}_{3}(p,q)$ is the third order revival operator defined below in \eqref{Revival Operator}. Then, the solution at rational time $t_{\mathrm{r}}= 2\pi\frac{p}{q}$ is given by
\begin{equation}
	\label{Airy Correspondence}
	u(x,t_{\mathrm{r}}) = e^{-i t_{\mathrm{r}}\theta^{3}} e^{i\theta x} \mathcal{T}_{3\theta^{2}t_{\mathrm{r}}} v^{(p,q)}(x,3\theta t_{\mathrm{r}}).
\end{equation}
Here $\mathcal{T}_{s}$ is the translation operator (see \eqref{Periodic Translation Operator}) and  $v^{(p,q)}(x,t)$ denotes the solution of the periodic problem for Schr\"{o}dinger equation with initial condition $v^{(p,q)}_{0}$.
\end{theorem}
It is clear from representation \eqref{Airy Correspondence} that we can expect revivals for the Airy quasi-periodic problem only when $\theta\in\mathbb{Q}$, see Corollary~\ref{Airy QP Revival}. Indeed, if $\theta\not\in\mathbb{Q}$, then the time $3 \theta t_{\mathrm{r}}$ is an irrational time for the solution of a periodic problem of the Schr\"{o}dinger equation, which is therefore a continuous function of $x$. We are not aware of any previous  result  in the literature concerning the failure of any form of revival to hold for a linear dispersive PDEs with coupling boundary conditions.

\smallskip
We devote the final part of the paper to the linear Schr{\"o}dinger equation \eqref{linS} with the Robin type boundary conditions \eqref{rbc}. In this case the boundary conditions do not couple the ends of the interval, in contrast with all other situations considered here. We show that only a weaker form of revival holds, leading us to reconsider what constitutes revival for a linear dispersive evolution equation. Specifically, we show that, while the solution is not given by finitely many translates of the initial condition, the presence of a periodic term in the solution representation guarantees that the dichotomy between the persistence versus regularisation of discontinuities at rational  versus generic time still holds. Our main statement can be formulated as follows.
\begin{theorem} 
\label{Robin Revival Corollary}
Consider the linear Schr\"{o}dinger equation \eqref{linS} with initial condition $u_{0} \in L^{2}(0,\pi)$ and Robin boundary conditions \eqref{rbc} with $b\not= 0, 1$. Let $p$ and $q$ be co-prime and $\mathcal{R}_2(p,q)$ be the second order revival operator defined below by \eqref{Revival Operator}. Let
\begin{equation*}
    f_{1}(x) = \sqrt{\frac{\pi}{2}}\frac{b}{(1-b)(e^{\frac{2\pi b}{1-b} - 1})} e^{\frac{b}{1-b}x}, \quad x\in (0,2\pi).
\end{equation*}
Then, the solution at rational time $t_{\mathrm{r}}=2\pi\frac{p}{q}$ is given by
	\begin{equation}
	\begin{aligned}
		\label{Robin Revival}
			u(x, t_{\mathrm{r}}) = &2\sqrt{\frac{2}{\pi}}
\langle u_{0},e^{\frac{b}{1-b}(\cdot)}\rangle_{L^{2}(0,\pi)} e^{i\frac{b^{2}}{(1-b)^{2}}t_{\mathrm{r}}} f_{1}(x) +
\mathcal{R}_2(p,q)\left[u_0^+(x)\right] \\ &+ \mathcal{R}_2(p,q)\left[2f_1*(u_0^--u_0^+))(x)\right], \quad x\in(0,\pi).
\end{aligned}
\end{equation}
where $u_0^{\pm }(x)$ are the even/odd extension in $(0,2\pi)$ of the initial condition, and $*$ denotes the $2\pi$-periodic convolution.
\end{theorem}

We conjecture that this weaker form of revival is generic in the case of boundary conditions that do not couple the interval endpoints. Our observations in this case complement those reported in \cite{boulton2020new} illustrating a new kind of revival phenomenon.

\subsection*{Periodic revival}

The original terminology seems to have originated from the experimentally observed phenomenon of \emph{quantum revival}~\cite{berry2001quantum, vrakking, yeazell}. This describes how an electron that is initially concentrated near a single location of its orbital shell is found concentrated again,  at certain specific times, near a finite number of orbital locations.  This led pure mathematicians to pose the question in terms of whether a quantum particle {\em knows} the time, \cite{kapitanski1999does}. 

A precursor of  the phenomenon was observed as far back as 1834 in optical experiments performed by Talbot \cite{talbot1836lxxvi}.  This motivated the pioneering work of Berry and collaborators \cite{berry1996quantum, berry1996integer, berry2001quantum}, on what they called the {\em Talbot effect} in the context of the linear free space Schr\"odinger equation. The concept was extended later to a class of linear dispersive equations that included the linearised Korteweg--deVries equation, first by Oskolkov, \cite{oskolkov1992class} and subsequently rediscovered  by Olver \cite{olver2010dispersive}, who called the effect \emph{dispersive quantisation}. It was later extended by Erdo\u{g}an and Tzirakis (see the monograph \cite{erdougan2016dispersive} and references therein). An exhaustive introduction to the history and context of this phenomena can be found in the recent survey \cite{smith2020revival}.

Questions have also been addressed on the fractal dimension of the solution profile at irrational times, hence almost everywhere in time, some of them resolved by Rodnianski in \cite{rodnianski2000fractal}. In a different direction, Olver and Chen in \cite{chen2013dispersion} and \cite{chen2014numerical} observed and confirmed numerically the revival and fractalisation effect in some non-linear, integrable and non-integrable, evolution problems. A number of their observations have been rigorously confirmed in \cite{erdogan2013talbot}, \cite{chousionis2014fractal}, \cite{erdougan2019fractal} by Erdo\u{g}an, Tzirakis, Chousionis and Shakan, and more recently in \cite{boulton2020new} for linear integro-differential dispersive equations.

The direct link between our present findings and this periodic framework can be put into perspective by following \cite[\S2.3]{erdougan2016dispersive}, as we briefly summarise. 

Consider general linear dispersive equations of the form
\begin{equation}
	\label{Linear Evolution Problems}
	u_{t}(x,t) +i P(-i\partial_{x})u(x,t)=0,\quad x\in [0,2\pi],\ t>0.
\end{equation}
where $P(\cdot)$ is a polynomial of degree $n$ with integer coefficients. Consider purely periodic boundary conditions, \emph{i.e.} \eqref{qpbc} with $\beta=1$. For initial datum $u(x,0)=u_0(x)\in\,L^2(0,2\pi) = L^{2}$, the solution   is given in terms of the eigenfunction expansion
\begin{equation}
	\label{Linear Evolution Problems solution}
	u(x,t) =\sum_{m\in\mathbb Z}\widehat{u_{0}}(m)e^{-iP(m) t}e_m(x),\qquad  \widehat{u_{0}}(m)=\langle u_0,e_m\rangle,
\end{equation}
where 
\begin{equation}
 	\label{Periodic Eigenpairs 2}
 	e_{m}(x) = \frac{e^{imx}}{\sqrt{2\pi}},\qquad  \langle f,g\rangle=\int_0^{2\pi} f(x)\overline{g(x)} dx,\;\;f,g\in\,L^2.
 \end{equation}
The family  $\{e_m\}_{m\in\mathbb Z}$ is  the orthonormal family of eigenfunctions of the self-adjoint periodic operator $P(-i\partial_{x})$. Note that the latter has a compact resolvent. If $u_0$ is continuous and periodic, the expression \eqref{Linear Evolution Problems solution} is also a continuous periodic function of $x$ and $t$. 
The case of equations \eqref{linS} and \eqref{airy}, corresponding to $P(k)=k^n$ with $n=2,3$, are among the simplest linear evolution equations, 
but they are important as they also appear as the linear part of important nonlinear PDEs of mathematical physics, namely the nonlinear Schr\"odinger and KdV equations respectively.

We focus now on the countable set of rational times, defined as follows.
\begin{definition}
We say that $t>0$ is a \emph{rational time} for the evolution problem \eqref{Linear Evolution Problems} if there exist co-prime, positive integers $p,q\in \mathbb N$ such that 
\begin{equation}\label{Rational Time}
t=2\pi \frac p q.
\end{equation}
\end{definition}

A self-contained proof of the following general result can be found in \cite[Theorem~2.14]{erdougan2016dispersive}.
This result says that, at these rational times, the solution of the periodic problem for \eqref{Linear Evolution Problems} has an explicit representation in terms of translates of $u_0$.

\begin{theorem}[Periodic Revival]
	\label{ET Theorem}
Consider equation \eqref{Linear Evolution Problems}, with initial condition  $u(x,0)=u_0(x)\in L^{2}$ and purely periodic boundary conditions, \eqref{ppbc} with all $\beta_{k}=1$. 
At rational time $t$ given by \eqref{Rational Time}
the solution $u(x,t)$ is given by
\begin{equation}
	\label{ET Formula}
	u\left(x,2\pi\frac{p}{q}\right) = \frac{1}{q}\sum_{k=0}^{q-1} G_{p,q}(k) u_{0}^{\ast} \left(x - 2\pi \frac{k}{q}\right),
\end{equation}
where $u_{0}^{*}$ is the $2\pi$-periodic extension of $u_{0}$, see \eqref{Periodic Extension} below. The coefficients $G_{p,q}(k)$ are given by  
\begin{equation}
	\label{ET Gauss Sum}
	G_{p,q} (k) = \sum_{m=0}^{q-1}e^{-2\pi iP(m)\frac{p}{q}} e^{2\pi i m\frac{k}{q}}.
\end{equation}
\end{theorem}

Note that the functions $G_{p,q} (k)$ in \eqref{ET Gauss Sum} are periodic number-theoretic functions $(\operatorname{mod}q)$, \emph{c.f.} \cite[\S27.10]{nist2010} of Gauss type,  but they are not Gauss sums, as the coefficients $e^{-2\pi iP(m)\frac{p}{q}}$ are not Dirichlet characters. 

The representation given in Theorem \ref{ET Theorem} describes explicitly the ``revival''  of the initial condition at rational times, as translated copies of it which are the building blocks of the solution representation.  
This is in contrast with the behaviour at generic, irrational times. For such times, the solution is continuous and indeed can be shown to have fractal behaviour as soon as the initial condition is sufficiently rough. To be more precise, the following and Theorem~\ref{ET Theorem} complement one another for the case of \eqref{linS}, see \cite{rodnianski2000fractal}.

\begin{theorem}\label{Fractalisation LS}
Let $P(k)=k^2$ and
assume that the hypotheses of Theorem~\ref{ET Theorem} hold true. Assume, additionally, that $u_0(x)$ is of bounded variation. Then, the solution for any value of $t$ that is not of the form \eqref{Rational Time} is a continuous function of $x$. Moreover, if
\[
u_0\notin \bigcup_{s>1/2} H^s(0,2\pi),
\]
where $H^s$ denotes the standard Sobolev space of order $s$,  then for almost every $t$, the solution is nowhere differentiable, and the graph of the real part of the solution has fractal dimension $3/2$.
\end{theorem}

A similar result holds in general for equation \eqref{Linear Evolution Problems},  see \cite{erdougan2016dispersive}.

\section{Revival operators}

This section is devoted to the notion of \emph{revival operators}, that can be regarded as the basic building blocks of the revival formula \eqref{ET Formula} for solutions of linear dispersive PDEs whose polynomial dispersion is of the form $P(k)=k^l$, $l\in\mathbb N$. They provide a compact notation, e.g. for the statement of Theorem~\ref{ET Theorem}. More significantly, while it is straightforward  to compute the corresponding Fourier representation, they give the crucial link for extending the revival results to more general pseudo-periodic problems from known cases, such as the linear Schr\"odinger equation, to higher order case, in particular the Airy equation.

Here and everywhere below, we will denote by $f^{\ast}$ the $2\pi$-periodic extension to $\mathbb R$ of a function $f$ defined on $[0,2\pi]$. Explicitly, 
\begin{equation}
	\label{Periodic Extension}
	f^{\ast} (x) = f(x - 2\pi m), \qquad 2\pi m\leq x < 2\pi (m+1), \quad m\in\mathbb{Z}.
\end{equation}

Because of the role of specific translation operators in what follows, we set our notation with the following definition. 
\begin{definition}[Periodic translation operator]
Let $s\in\mathbb{R}$.
The  \emph{periodic translation operator} $\mathcal{T}_{s}:L^{2} \rightarrow L^{2}$ is  given by 
\begin{equation}
	\label{Periodic Translation Operator}
	\mathcal{T}_{s}f(x) = f^{\ast} (x-s), \qquad  x\in[0,2\pi).
\end{equation}
\end{definition}

Note that $\mathcal{T}_s$ are isometries. In our scaling, the Fourier coefficients of $\mathcal{T}_{s}f$, $f\in L^2$,  turn out to be 
\begin{equation}
	\label{Periodic Translation Operator FC}
	\widehat{\mathcal{T}_{s}f}(m) = 
	 \int_{0}^{2\pi} \mathcal{T}_{s}f(x) \overline{e_m(x)}dx = e^{-ims} \widehat{f}(m).
\end{equation}
Revival operators are formed as finite linear combinations of specific translation operators. 

\begin{definition}[Periodic revival operator]
	\label{Periodic Revival Operator}
	Let $p$ and $q$ be integers and co-prime. Let $\ell \in \mathbb{N}$. The \emph{periodic revival operator} $\mathcal{R}_{\ell}(p,q): L^{2}\rightarrow L^2$ of order $\ell$ at $(p,q)$ is given by 
	\begin{equation}
		\label{Revival Operator}
		\mathcal{R}_{\ell}(p,q) f = \frac{\sqrt{2\pi}}{q}\sum_{k=0}^{q-1} G^{(l)}_{p,q}(k)\mathcal{T}_{\frac{2\pi k}{q}}f, \quad G^{(l)}_{p,q}(k)=\sum_{m=0}^{q-1} e^{-im^{\ell}\frac{2\pi p}{q}} e_{m}\left(\frac{2\pi k}{q}\right),
	\end{equation}
where $e_{m}(x)$ are the normalised eigenfunctions of the $2\pi$-periodic problem given in \eqref{Periodic Eigenpairs 2}.
\end{definition}

As we shall see next, from the Fourier representation, it follows that all periodic revival operators are isometries.  

\begin{lemma}
	\label{Revival Operator Lemma}
	Let $p$ and $q$ be integers and co-prime. Let $\ell \in \mathbb{N}$. Then, $\mathcal{R}_{\ell}(p,q)$ given by \eqref{Revival Operator} is an isometry of $L^2$. Moreover, for all $f \in L^2$ we have 
	\begin{equation}
		\label{Revival Operator Lemma 1}
		\langle \mathcal{R}_{\ell}(p,q)f,e_{j}\rangle = e^{-ij^{\ell}\frac{2\pi p}{q}} \widehat{f}(j).
	\end{equation}
\end{lemma}
\begin{proof}
	In order to deduce that $\mathcal{R}_{\ell}(p,q)$ is an isometry on $L^2$, it is enough to prove \eqref{Revival Operator Lemma 1}. From the left hand side of \eqref{Revival Operator} and from \eqref{Periodic Translation Operator FC}, it follows that
	\begin{equation*}
		\begin{aligned}
			\langle \mathcal{R}_{\ell}(p,q)f,e_{j}\rangle 
			= \frac{\sqrt{2\pi}}{q} \sum_{k=0}^{q-1} G^{(l)}_{p,q}(k) \langle \mathcal{T}_{\frac{2\pi k}{q}}f,e_{j}\rangle = \widehat f(j)\frac{\sqrt{2\pi}}{q}  \sum_{k=0}^{q-1}e^{-ij\frac{2\pi k}{q}}G^{(l)}_{p,q}(k).\end{aligned}
	\end{equation*} 
	Subtitute the right hand side of \eqref{Revival Operator} to get,
	\begin{equation*}
		\langle \mathcal{R}_{\ell}(p,q)f,e_{j}\rangle = \frac{\widehat{f}(j)}{q} \sum_{k=0}^{q-1}
		\sum_{m=0}^{q-1} e^{-im^{\ell}\frac{2\pi p}{q}} e^{i(m-j)\frac{2\pi k}{q}}=\frac{\widehat{f}(j)}{q} \sum_{m=0}^{q-1}e^{-im^{\ell}\frac{2\pi p}{q}}
		\sum_{k=0}^{q-1}  e^{i(m-j)\frac{2\pi k}{q}}.
	\end{equation*}
	
	Now, if $m\not\equiv j$ $(\operatorname{mod} q)$, then there exists $z\in\mathbb{Z}$ not a multiple of $q$, such that $m-j = z_{1}q+z$ for $z_{1}\in\mathbb{Z}$. Hence,
	\begin{equation*}
		\sum_{k=0}^{q-1} e^{i(m-j)\frac{2\pi k}{q}} = \sum_{k=0}^{q-1} e^{iz_{1}q\frac{2\pi k}{q}} \ e^{iz\frac{2\pi k}{q}} = \sum_{k=0}^{q-1}\left(e^{i2\pi\frac{z}{q}}\right)^{k} = \frac{1 - \left(e^{i2\pi\frac{z}{q}}\right)^{q}}{1-e^{i2\pi\frac{z}{q}}} = 0.
	\end{equation*}
	On the other hand, whenever $m \equiv j$ $(\operatorname{mod}q)$, we have $m-j = z_2q$ for $z_2\in\mathbb{Z}$ and so
	\begin{equation*}
		\sum_{k=0}^{q-1} e^{i(m-j)\frac{2\pi k}{q}} = \sum_{k=0}^{q-1} e^{iz_2q\frac{2\pi k}{q}} = q.
	\end{equation*}
	Moreover, in this case we know that, for any $\ell\in\mathbb{N}$, $m^{\ell} \equiv j^{\ell}$ $(\operatorname{mod} q)$, and so $m^{\ell} = j^{\ell} + z_3q$ for some other $z_3\in\mathbb{Z}$, hence
	\begin{equation*}
		e^{-im^{\ell}\frac{2\pi p }{q}} = e^{-ij^{\ell}\frac{2\pi p }{q}} e^{-i z_3q \frac{2\pi p}{q}} = e^{-ij^{\ell}\frac{2\pi p }{q}}.
	\end{equation*}
	Therefore, as $m$ runs from $0$ to $q-1$, we find that
	\begin{equation*}
		\langle \mathcal{R}_{\ell}(p,q)f,e_{j}\rangle = \widehat{f}(j) e^{-ij^{\ell}\frac{2\pi p}{q}} ,
	\end{equation*}
as claimed.
\end{proof}
The proof of the lemma above relies on elementary arguments and depends on the specific form of the eigenfunctions $e_m(x)$ and their periodicity. This is in fact at the heart of the periodic revival phenomenon. It suggests strongly that  such phenomenon depends crucially on periodicity and will not survive if other boundary conditions are prescribed. The investigation of the validity of this statement is the motivation for this work.   

By immediate substitution, Theorem~\ref{ET Theorem} applied to the linear Schr\"{o}dinger  and  Airy equations can be reformulated in terms of revival operators.

\begin{lemma}
\label{Sch Airy Periodic Revivals Lemma}
Let $u_{0}\in L^2$ and assume periodic boundary conditions, \eqref{qpbc} with $\beta=1$. At rational time $t=2\pi\frac p q$, the solution to the periodic problem for equation \eqref{linS} starting at $u_0$ is given by 
\begin{equation}\label{R2ls}
	u\left(x,2\pi\frac{p}{q}\right) = \mathcal{R}_{2}(p,q)u_{0}(x)
	\end{equation}
	and the solution to the periodic problem for equation \eqref{airy} starting at $u_0$ is given by 
\begin{equation}\label{R2ai}
	u\left(x,2\pi\frac{p}{q}\right) = \mathcal{R}_{3}(p,q)u_{0}(x).
	\end{equation}
\end{lemma}

\section{Pseudo-periodic problems for the linear Schr\"odinger equation}\label{pseudo ls}

In this section we give an alternative proof of the results reported in \cite{olver2018revivals}, by deriving a new representation of the solution of the problem \eqref{linS}-\eqref{ppbc}, namely
\begin{equation}
	\label{Schrodinger Pseudo-periodic Problem}
	\begin{aligned}
		&iu_t+u_{xx}=0, \quad u(x,0) = u_{0}(x)\in\, L^2,\\
		&\beta_{0} u(0,t) = u(2\pi,t), \quad \beta_{1} u_x(0,t) =u_x(2\pi,t),
	\end{aligned}
\end{equation}
where $\beta_{0}, \ \beta_{1}\in\mathbb{C}$ satisfy
\[
   \arccos\left(\frac{1+\beta_0\beta_1}{\beta_0+\beta_1}\right)\in\mathbb{R}.
\]
The latter condition ensures that all the eigenvalues of the underlying (closed) spatial operator are real and that this operator has a family of eigenfunctions which is complete in $L^2$, forming a bi-orthogonal basis. Moreover, this family reduces to an orthonormal basis, i.e. the operator is self-adjoint,  if and only if $\overline{\beta_0}\beta_1=1$.  For details, see \cite{olver2018revivals}.

Our goal is to show that the solution of \eqref{Schrodinger Pseudo-periodic Problem} can be written as the sum of four terms,  each obtained as the solution of a  {\em periodic} problem. These four periodic problems start from an initial condition obtained by a suitable transformation of the given initial  $u_{0}(x)$. 

In order to construct a solution of \eqref{Schrodinger Pseudo-periodic Problem}, we consider the  bi-orthogonal basis $\{\phi_{j},\psi_{\ell}\}_{j,\ell\in\mathbb{Z}}$ formed by the eigenfunctions of the spatial operator and their adjoint pairs. The spectral problem is given by
\begin{equation}
	\label{Non-Self-adjoint eigenvalue problem}
	-\phi''(x) = \lambda \phi(x), \quad \beta_{0}\phi(0)=\phi(2\pi), \ \beta_{1}\phi'(0)=\phi'(2\pi).
\end{equation}  
As shown in \cite{olver2018revivals},  the eigenvalues $\{\lambda_j\}_{j\in \mathbb{Z}}$ are given by 
\begin{equation}
	\label{Non-Self-adjoint Eigenvalues}
	\lambda_{j} = k_{j}^{2}, \quad k_{j} = (j+k_{0}), \quad k_{0} = \frac{1}{2\pi}\arccos(\frac{1+\beta_{0}\ \beta_{1}}{\beta_{0}+\beta_{1}}).
\end{equation}
and the corresponding eigenfunctions are
\begin{equation}
	\label{Non-Self-adjoin eigefunctions}
	\phi_{j}(x) = \frac{1}{\sqrt{2\pi \tau}}(e^{ik_{j}x} + \Lambda_{0}e^{-ik_{j}x}),
	\end{equation}
where
\begin{equation}
	\label{Non-Self-adjoint Various CST}
	\tau = \frac{(\gamma^2 + 1)(\beta_{0}\beta_{1}+1)-2\gamma(\beta_{0}+\beta_{1})}{(\beta_{0} \gamma - 1)(\beta_{1}\gamma - 1)}, \ \Lambda_{0} = \frac{\gamma - \beta_{0}}{\beta_{0} - \gamma^{-1}}= \frac{\gamma - \beta_{1}}{\gamma^{-1} - \beta_{1}},
\end{equation}
and
\begin{equation}
	\label{Non-Self-Adjoint Gamma}
	\gamma = e^{ik_{j}2\pi} = e^{i2\pi k_{0}} =  \frac{1+\beta_{0}\ \beta_{1}}{\beta_{0}+\beta_{1}} + i \sqrt{1 -  \left(\frac{1+\beta_{0}\ \beta_{1}}{\beta_{0}+\beta_{1}}\right)^{2}}.
\end{equation}
We require also the eigenfunctions of the adjoint spectral problem
\begin{equation}
	\label{Dual eigenvalue problem}
	-\psi''(x) = \lambda \psi(x), \quad \psi(0)=\bar{\beta_{1}}\psi(a), \ \psi'(0)=\bar{\beta_{0}}\psi'(a).
\end{equation}  
These are given by
\begin{equation}
	\label{Dual eigefunctions}
	\psi_{j}(x) = \frac{1}{\sqrt{2\pi\tau}}(e^{ik_{j}x} + I_{0}e^{-ik_{j}x}) 
	\end{equation}
where $\tau$ is as in \eqref{Non-Self-adjoint Various CST} and
\begin{equation}
\label{Dual Various CST}
	 I_{0} = \frac{\gamma - 1/\bar{\beta_{1}}}{1/\bar{\beta_{1}} - \gamma^{-1}}.
	 \end{equation}
The family  $\{\phi_{j}\}_{j\in\mathbb{Z}}$ is a complete system of $L^2$. Then, for any fixed time $t\geq 0$ and initial $v_{0}\in L^2$, the solution to \eqref{Schrodinger Pseudo-periodic Problem} is given by the spectral expansion
\begin{equation}
	\label{Non-Self-Adjoint Solution}
	\begin{aligned}
	&u(x,t) = \sum_{j\in\mathbb{Z}} \langle u_{0} , \psi_{j}\rangle e^{-ik_{j}^{2}t} \phi_{j} (x) \\
	& = \frac{1}{2\pi\tau}\sum_{j\in\mathbb{Z}} \Big(\int_{0}^{2\pi}u_{0}(y)e^{-ik_{j}y} dy + \bar{I}_{0}\int_{0}^{2\pi} u_{0}(y)e^{ik_{j}y}dy\Big) e^{-ik_{j}t^{2}} \big(e^{ik_{j}x} + \Lambda_{0} e^{-ik_{j}x}\big).
	\end{aligned}
\end{equation}

Our alternative proof that this problem exhibits the periodic revival phenomenon will rely on the existence of revivals for suitable periodic problems. Given  $u_{0}\in L^2$, we define $v_{0}$, $w_{0}\in \,L^2$ as
	\begin{equation}
		\label{Non-Self-adjoint initial condition}
		v_{0} (x) = u_{0}(x) e^{-i  k_{0}x}, \quad w_{0}(x) = u_{0}(x) e^{i k_{0}x},
	\end{equation}
where $k_0\in\mathbb{R}$ is defined in \eqref{Non-Self-adjoint Eigenvalues}.  For any $f\in L^2$, we will denote by the symbol $f^{\natural}(x)$ the reflection of $f(x)$ with respect to $x=\pi$, namely
\begin{equation}
		\label{Reflected initial condition}
		f^{\natural}(x) = f(2\pi-x).
	\end{equation}

\begin{proposition}
	\label{Non-Self-Adjoint Correspondence Theorem}
The solution $u(x,t)$ of \eqref{Schrodinger Pseudo-periodic Problem}  admits the following representation,
	\begin{equation}
		\label{Non-Self-adjoint Correspondence}
		\begin{aligned}
			u(x,t) = \frac{e^{-i k_{0}^{2}t}}{\tau} \Big\{& e^{ik_{0}x} \mathcal{T}_{2k_0t} v(x,t) + \Lambda_{0} e^{-ik_{0}x} \mathcal{T}_{-2k_0t} v^{\natural}(x,t) \\
			& + \bar{I}_{0} e^{ik_{0}x} \mathcal{T}_{2k_0t} w^{\natural}(x,t) + \Lambda_{0}\bar{I}_{0} e^{-ik_{0}x} \mathcal{T}_{-2k_0t} w(x,t) \Big\},
		\end{aligned}
	\end{equation}
	where $\mathcal{T}_{s}$ is the translation operator defined by \eqref{Periodic Translation Operator},  the constants $\tau$, $\Lambda_{0}$ are given in \eqref{Non-Self-adjoint Various CST} and  $I_{0}$ by \eqref{Dual Various CST}. Here
$v,w,v^{\natural},w^{\natural}$ are
	 the  solutions of the periodic problem, \emph{i.e.} $\beta_0=\beta_1=1$, with initial conditions as follows, 
\begin{itemize}
\item
 $v(x)$ denotes the solution corresponding to  initial condition $v_0(x)$
\item
 $w(x)$ denotes the solution corresponding to  initial condition $w_0(x)$
 \item
 $v^{\natural}(x)$ denotes the solution corresponding to  initial condition $v^{\natural}_0(x)$
\item
 $w^{\natural}(x)$ denotes the solution corresponding to  initial condition $w^{\natural}_0(x)$.
 \end{itemize}
	\end{proposition}

Before giving a proof, we highlight the important consequence of this proposition. Substituting the expression \eqref{R2ls} for the solution of the periodic problem in \eqref{Non-Self-adjoint Correspondence}, one obtains revival for the pseudo-periodic linear Schr\"odinger equation. 

\begin{corollary}[Pseudo-periodic revival property]
	\label{Non-Self-adjoint Pseudo-periodic Revival}
	The solution of the pseudo-periodic problem \eqref{Schrodinger Pseudo-periodic Problem} at rational times, is given by
	\begin{equation}
		\label{Non-Self-adjoint Revival}
\begin{aligned}
			u\left(x, 2\pi\frac{p}{q}\right) = &\frac{e^{-i\frac{2\pi k_{0}^{2}p}{q}}}{\tau} \Big \{ 
			e^{i k_{0}x} \left[\mathcal{T}_{\frac{4\pi k_{0}p}{q}}\mathcal{R}_{2}(p,q)\right]
			e^{-i  k_{0}x}u_{0}(x) \\
			&+e^{-ik_{0}x}\left[ \Lambda_{0} e^{-i  k_{0}2\pi} \mathcal{T}_{-\frac{4\pi k_{0}p}{q}}  \mathcal{R}_{2}(p,q)\right]
			e^{i  k_{0}x}u_{0}^{\natural}(x) \\
			& +  e^{i k_{0}x}\left[\bar{I}_{0}e^{i  k_{0}2\pi} \mathcal{T}_{\frac{4\pi k_{0}p}{q}} \mathcal{R}_{2}(p,q)\right]e^{-i  k_{0}x}u_{0}^{\natural}(x)
			\\
			&+ e^{-i k_{0}x} \left[ \Lambda_{0}\bar{I}_{0}\mathcal{T}_{-\frac{4\pi k_{0}p}{q}}  \mathcal{R}_{2}(p,q)\right]e^{i  k_{0}x}u_{0}(x)
			\Big\}.
		\end{aligned}
	\end{equation}
\end{corollary}

\begin{remark}\label{v0expr}
In expression (\ref{Non-Self-adjoint Revival}), the solution is given explicitly in terms of a finite number of translated copies of $u_0(x)e^{\pm ik_0x}$. Note that the final result is then multiplied by $e^{\mp ik_0x}$, and hence the solution is indeed given in terms of a finite linear combination of translated copies of $u_0(x)$. This can be verified by substituting the expression for $\mathcal{R}_{2}(p,q)$ in the first part of formula (\ref{Non-Self-adjoint Revival}), to obtain
$$
e^{i k_{0}x} \mathcal{T}_{\frac{4\pi k_{0}p}{q}}\mathcal{R}_{2}(p,q)
			e^{-ik_0x}u_{0}(x)=e^{ik_0^2\frac{4\pi p}{q}}\mathcal{T}_{\frac{4\pi k_{0}p}{q}}\tilde{\mathcal{R}}_{2}(p,q)u_0(x),
$$
where  $\tilde{\mathcal{R}}_{2}(p,q)$ differs from  ${\mathcal{R}}_{2}(p,q)$ only in that each term $e_{m}(\frac{2\pi k}{q})$ is replaced by $e_{m}(\frac{2\pi k}{q} (1+\frac{k_0}{m}))$. The other three terms in the expression  (\ref{Non-Self-adjoint Revival}) for the solution can be handled similarly. 
\end{remark}

\begin{proof}[Proof of Proposition \ref{Non-Self-Adjoint Correspondence Theorem}]
 Consider each of the terms in the series \eqref{Non-Self-Adjoint Solution}. Using the definition (\ref{Non-Self-adjoint initial condition}) of $v_0$ and $w_0$,  we have   
	\begin{equation}
		\label{NSAT2}
		\begin{aligned}
			\int_{0}^{2\pi}u_{0}(y)\frac{e^{-ik_{j}y}}{\sqrt{2\pi}}dy + \bar{I}_{0}\int_{0}^{2\pi} u_{0}(y)\frac{e^{ik_{j}y}}{\sqrt{2\pi}}dy= \widehat{v_{0}}(j) + \bar{I}_{0}\widehat{w_{0}}(-j).
		\end{aligned}
	\end{equation}
Recall that $k_j=k_0+j$. Moreover, we have the elementary but key relation,
	\begin{equation}
		\label{NSAT3}
		e^{-ik_{j}^{2}t} = e^{-i k_{0}^{2}t} \ e^{-2k_0jt} \ e^{-ij^{2}t}
	\end{equation}
	and for the eigenfunctions
	\begin{equation}
		\label{NSAT4}
		\frac{e^{ik_{j}x}}{\sqrt{2\pi}} + \Lambda_{0} \frac{e^{-ik_{j}x}}{\sqrt{2\pi}} = e^{ i k_{0}x} e_{j}(x) + \Lambda_{0} e^{-ik_{0}x}e_{-j}(x).
	\end{equation} 
Here the $e_j(x)$ are the periodic eigenfunctions.

By substituting \eqref{NSAT2}, \eqref{NSAT3} and \eqref{NSAT4} in \eqref{Non-Self-Adjoint Solution} we obtain
	\begin{equation}
		\label{NSAT5}
		\begin{aligned}
			u(x,t) =  \frac{e^{-ik_{0}^{2}t}}{\tau} &\sum_{j\in\mathbb{Z}} e^{-i2k_0jt} e^{-ij^{2}t} \Big( e^{ik_{0}x}\widehat{v_{0}}(j) e_{j}(x) + \Lambda_{0} e^{-ik_{0}x}\widehat{v_{0}}(j) e_{-j}(x) \\
			& + \bar{I}_{0} e^{ik_{0}x}\widehat{w_{0}}(-j) e_{j}(x) + \Lambda_{0} \bar{I}_{0} e^{-ik_{0}x}\widehat{w_{0}}(-j) e_{-j}(x)\Big).
		\end{aligned}
	\end{equation} 
	Each term in  \eqref{NSAT5}  is the solution of a periodic problem. Indeed, from \eqref{Periodic Translation Operator FC} it follows that for $f \in\, L^2$,
$\mathcal{T}_sf(x)=\sum_{j\in\mathbb Z}e^{-ijs} \widehat{f}(j)e_j(x)$, hence we have
	\begin{equation}
		\label{NSAT6} 
		\sum_{j\in\mathbb{Z}} e^{-i2k_0jt} e^{-ij^{2}t} e^{ik_{0}x}\widehat{v_{0}}(j)e_{j}(x) = e^{ik_{0}x} 	\mathcal{T}_{2k_0t}	\Big(\sum_{j\in\mathbb{Z}}\widehat{v_{0}}(j)e^{-ij^{2}t}e_{j}(x) \Big) = e^{ik_{0}x} \mathcal{T}_{2k_0t} v(x,t),
	\end{equation}
	where $v(x,t)$ solves the periodic equation with initial condition $v_{0}(x)$. Similar calculation for the remaining terms yields the representation \eqref{Non-Self-adjoint Correspondence}.
\end{proof}

Note that for the self-adjoint case, $\beta_0\bar{\beta_1}=1$, the following reduction of \eqref{Non-Self-adjoint Correspondence} is valid,
\begin{equation}
		\label{Self-adjoint Correspondence}
		\begin{aligned}
			u(x,t) = \frac {e^{-i k_{0}^{2}t}}{1+|\Lambda_0|^2}\Big\{& e^{ik_{0}x} \mathcal{T}_{2k_0t} v(x,t) + \Lambda_{0} e^{-ik_{0}x} \mathcal{T}_{-2k_0t} v^{\natural}(x,t) \\
			& + \bar{\Lambda}_{0} e^{ik_{0}x} \mathcal{T}_{2k_0t} w^{\natural}(x,t) +| \Lambda_{0}|^2e^{-ik_{0}x} \mathcal{T}_{-2k_0t} w(x,t) \Big\},
		\end{aligned}
	\end{equation}
with all notation as in Proposition \ref{Non-Self-Adjoint Correspondence Theorem}.

 \subsection{The quasi-periodic  case}
We now describe the specific form of the solution of the quasi-periodic boundary value problem for \eqref{linS}, corresponding to $\beta_0=\beta_1=\beta$ in \eqref{Schrodinger Pseudo-periodic Problem}. This specific case appears to be of importance for the study of the vortex filament equation with non-zero torsion \cite{de2020evolution}. The self-adjoint case corresponds to $|\beta|^2=1$ and it has been studied in the context of quantum revivals,  as well as experimentally, in \cite{xue2014observation}. 

Set $\beta = e^{2\pi i\theta}$ for $\theta\in (0,1)$ in \eqref{Schrodinger Pseudo-periodic Problem}.  For $k_0$ and $\Lambda_0$ as in (\ref{Non-Self-adjoint Various CST}), we have
$$
\cos (2\pi k_0)=\frac {1+\beta^2}{2\beta}=\frac{1+e^{4\pi i \theta}}{2e^{2\pi i\theta}}=\cos(2\pi \theta).
$$
So we pick
$$
 k_0=\theta,\qquad 
\gamma=e^{2\pi i \theta}=\beta \quad \text{and} \quad \Lambda_0=\frac{\gamma-\beta}{\beta-\gamma^{-1}}=0.
$$
Substituting these values into \eqref{Self-adjoint Correspondence}, yields the significantly reduced expression,
\begin{equation}
		\label{quasi-periodic Self-adjoint Correspondence}
			u(x,t) = e^{-i\theta^{2}t} e^{i\theta x} \mathcal{T}_{2\theta t} v(x,t),
	\end{equation}
where $v(x,t)$ the solution of the periodic problem with initial condition $v_0(x)$ as in \eqref{Non-Self-adjoint initial condition}. In particular, at rational times we obtain the representation formula
\begin{equation}
\label{quasiperrev}
u\left(x,2\pi\frac pq\right)=e^{-i\theta^{2}2\pi\frac pq} e^{i\theta x} \mathcal{T}_{4\pi\theta \frac pq}\mathcal{R}_{2}(p,q)
			e^{-i\theta x}u_{0}(x).
			\end{equation}

\begin{remark}The comment made in Remark~\ref{v0expr}, applies also to the revival expression \eqref{quasiperrev}. The latter can also be obtained directly, by expanding the solution in terms of the eigenfunctions 
of the associated spatial operator and their adjoint pair.
\end{remark}

\section{Quasi-periodic problems for the Airy equation}\label{Airy's Quasi-Periodic Problem}

We now turn to the time evolution problem for the Airy equation with quasi-periodic boundary conditions, defined by  \eqref{airy}--\eqref{qpbc} with
$\beta= e^{i2\pi\theta}$ for $\theta\in[0,1)$, namely
\begin{equation}
	\label{Airy QPP}
	\begin{aligned}
	&u_t(x,t) - u_{xxx}(x,t)=0, \qquad u(x,0) = u_{0}(x), \\
	&e^{i2\pi\theta} \partial_{x}^{m} u(0,t) = \partial_{x}^{m}u(2\pi,t), \quad m = 0,1,2.
	\end{aligned}
\end{equation}
We give the proof of Theorem \ref{Airy Correspondence Theorem}, which describes the solution of \eqref{Airy QPP} in terms of the solution of a periodic problem for the linear Schr\"odinger equation.

The spatial operator $i\partial_x^3$ with the given boundary conditions is self-adjoint. Moreover, unlike the general quasi-periodic boundary conditions, we can find the eigenpairs of this operator explicitly. Because of this, it is possible to argue in similar fashion as in Section~\ref{pseudo ls}.  This leads to the conclusion that, in contrast to the  linear Schr\"odinger equation, it is not possible to establish a direct correspondence between the solution of \eqref{Airy QPP} and the solution of one or more periodic problems {\em evaluated  at the same time}. The correspondence that we establish in Theorem \ref{Airy Correspondence Theorem}, connects the solution of the Airy equation at a rational time $t_{\mathrm{r}}$  to the solution of an associated problem for the linear Schr\"odinger equation evaluated at a time $t_{\theta}$ that depends on $t_{\mathrm{r}}$ and on $\theta$. As a consequence, we show below that revivals for problem \eqref{Airy QPP} arise  if and only if $\theta\in\mathbb{Q}$. 

The eigenvalue problem is now given by
\begin{equation}
\label{Non-Self-adjoint eigenvalue problem airy}
	-\phi'''(x) = i\lambda \phi(x), \quad e^{i2\pi\theta}\phi(0)=\phi(2\pi), \,e^{i2\pi\theta}\phi'(0)=\phi'(2\pi),
	\, e^{i2\pi\theta}\phi''(0)=\phi''(2\pi).
\end{equation}  
Hence, it is straightforward to compute that the eigenvalues are given by 
\begin{equation}
	\label{Airy Eigenvalue}
	\lambda_{m} = k_{m}^{3}, \qquad k_{m} =m + \theta, \quad m\in \mathbb{Z}
\end{equation}
and the corresponding normalized eigenfunctions  by
\begin{equation}
	\label{Airy Eigenfunctions}
	\phi_{m} (x)= \frac{e^{ik_{m}x}}{\sqrt{2\pi}}=e^{i\theta x}e_m(x), \quad m\in \mathbb{Z}.
\end{equation}
Thus, for any fixed time $t\geq 0$ and initial $u_{0} \in L^2$, the solution to \eqref{Airy QPP} is 
\begin{equation}
	\label{Airy QP Solution}
	u(x,t) = \sum_{m\in\mathbb{Z}} \langle u_{0} , \phi_{m}\rangle e^{-ik_{m}^{3}t} \phi_{m} (x).
\end{equation}
We are now ready to prove Theorem~\ref{Airy Correspondence Theorem}.
\begin{proof}[Proof of Theorem \ref{Airy Correspondence Theorem}]
According to \eqref{Airy Eigenfunctions},
\begin{equation}
	\label{ACT3}
	\langle u_{0}, \phi_{j}\rangle = \int_0^{2\pi}u_0(x)e^{-i\theta x}\overline{e_j(x)}dx=\widehat{w_{0}}(j), \quad w_0(x)=u_0(x)e^{-i\theta x}.
\end{equation}
The exponential term $e^{-ik_{j}^{3}t_{\mathrm{r}}}$ can be written as 
\begin{equation}
	\label{ACT4}
	e^{-ik_{j}^{3}t_{\mathrm{r}}} = e^{-i(j+\theta)^{3}t_{\mathrm{r}}}=
	e^{-i\theta^{3}t_{\mathrm{r}}} e^{-i j^3t_{\mathrm{r}}} e^{-i j3\theta^2 t_{\mathrm{r}}} e^{-ij^{2}3\theta t_{\mathrm{r}}}.
\end{equation}
Substituting all this into \eqref{Airy QP Solution} for the solution of \eqref{Airy QPP}, we find
\begin{equation}
	\label{ACT5}
	\begin{aligned}
	u(x,t_{\mathrm{r}}) &= \sum_{j\in\mathbb{Z}} \langle u_{0} , \phi_{j}\rangle e^{-ik_{j}^{3}t_{\mathrm{r}}} \phi_{j} (x)\\
	&=\sum_{j\in\mathbb{Z}}\widehat{w_{0}}(j)  e^{-i\theta^{3}t_{\mathrm{r}}} e^{-i j^3t_{\mathrm{r}}} e^{-i j3\theta^2 t_{\mathrm{r}}} e^{-ij^{2}3\theta t_{\mathrm{r}}} e^{i\theta x}e_j(x)\\
	&=e^{-i\theta^{3}t_{\mathrm{r}}}e^{i\theta x}\sum_{j\in\mathbb{Z}}\widehat{w_{0}}(j)  e^{-i j^3t_{\mathrm{r}}} e^{-i j3\theta^2 t_{\mathrm{r}}} e^{-ij^{2}3\theta t_{\mathrm{r}}} e_j(x) \\
	& = e^{-i\theta^{3}t_{\mathrm{r}}}e^{i\theta x} \mathcal{T}_{3\theta^2t_{\mathrm{r}}} \Big(\sum_{j\in\mathbb{Z}}\widehat{w_{0}}(j)  e^{-i j^3t_{\mathrm{r}}}  e^{-ij^{2}3\theta t_{\mathrm{r}}} e_j(x) \Big).
	\end{aligned}
\end{equation}
For the last equality we have used the Fourier representation \eqref{Periodic Translation Operator FC} of the translation operator $\mathcal{T}_{s}$. 

Now, by virtue of Lemma~\ref{Revival Operator Lemma}, 
$$
\widehat{w_0}(j) e^{-ij^{3}t_{\mathrm{r}}}=\langle \mathcal{R}_{3}(p,q)w_0,e_{j}\rangle =\langle v^{(p,q)}_{0}, e_j\rangle=\widehat{v^{(p,q)}_{0}}(j),
$$
where the function $v^{(p,q)}_{0}(x)$ is given by \eqref{Airy correspondence IC}.  Substituting this final identity into \eqref{ACT5}, gives
\begin{equation}
	\label{ACT6}
		u(x,t_{\mathrm{r}}) =e^{-i\theta^{3}t_{\mathrm{r}}}e^{i\theta x} \mathcal{T}_{3\theta^2t_{\mathrm{r}}} \Big(\sum_{j\in\mathbb{Z}}\widehat{v^{(p,q)}_{0}}(j)   e^{-ij^{2}3\theta t_{\mathrm{r}}} e_j(x) \Big) 
		= e^{-i\theta^{3}t_{\mathrm{r}}}e^{i\theta x} \mathcal{T}_{3\theta^2t_{\mathrm{r}}}  v^{(p,q)}(x, 3\theta t_{\mathrm{r}}).
\end{equation}
as claimed.
\end{proof}

The fundamental difference with the case of the linear Schr\"odinger equation analysed in the previous section, lies in the fact that the solution of the quasi-periodic problem for the Airy equation corresponds to the solution of a suitable periodic problem but {\em evaluated at a different time}.  Indeed, Theorem \ref{Airy Correspondence Theorem} states that the solution of \eqref{Airy QPP} at time $t=t_{\mathrm{r}}$ is obtained via the solution of a periodic problem for the Schr\"odinger equation evaluated at time $t=3\theta t_{\mathrm{r}}$. If $\theta\notin\mathbb Q$, this is an irrational time, for which the fractalisation result of Theorem \ref{Fractalisation LS} applies. From this it follows that,  the quasi-periodic Airy problem exhibits revivals at rational times if and only if $\theta\in\mathbb Q$. To be more precise, we have the following two  possibilities.

\begin{enumerate}
\item \underline{Case $\theta\in\mathbb Q$}. The time $t=3\theta t_{\mathrm{r}}$  is a rational time for Schr\"{o}dinger's periodic problem. Hence Airy's quasi-periodic problem will exhibit revivals at any rational time $t_{\mathrm{r}}$.

\item \underline{Case $\theta\notin\mathbb Q$}. The time $t=3\theta t_{\mathrm{r}}$  is irrational for Schr\"{o}dinger's periodic problem. It follows that the solution of Airy's quasi-periodic problem at  rational times $t_{\mathrm{r}}$ is a continuous but nowhere differentiable function, and there is no revival at rational times in this case.
\end{enumerate}

We now establish a representation formula for the solution at rational times, which implies the validity of the revival phenomenon observed in \eqref{Airy QPP} in the case $\theta\in\mathbb Q$. The proof of the next statement is a direct consequence of combining Theorem~\ref{Airy Correspondence Theorem} with Lemma~\ref{Sch Airy Periodic Revivals Lemma}.

\begin{corollary}[Quasi-Periodic Revival] \label{Airy QP Revival}
Let $(p,q),\,(c,d)$ be pairs of co-prime positive integers, with $c<d$. Set $\theta_{\mathrm{r}}= c/d <1$. Let $u_{0}\in L^2$. For $\theta=\theta_{\mathrm{r}}$, the solution $u(x,t)$ of the linear Airy equation with  initial condition given by $u_0$ and quasi-periodic boundary conditions \eqref{Airy QPP}, at rational time $t_{\mathrm{r}} = 2\pi\frac{p}{q}$ is given by,
\begin{equation}
	\label{Airy QP Revival Formula}
	u(x,t_{\mathrm{r}}) = e^{i\theta_{\mathrm{r}} x} \left[e^{-i \theta_{\mathrm{r}}^{3}t_{\mathrm{r}}}  \mathcal{T}_{3\theta_{\mathrm{r}}^2t_{\mathrm{r}}} \mathcal{R}_{2}(3cp,dq)  \mathcal{R}_{3}(p,q)\right] e^{-i\theta_{\mathrm{r}} x}u_{0}(x).
\end{equation}
\end{corollary}

The comment made in Remark~\ref{v0expr}, applies also to the revival expression  \eqref{Airy QP Revival Formula}. Indeed, the latter has an alternative representation  in terms of the eigenfunctions $\phi_m(x)$ of the spatial quasi-periodic operator given by \eqref{Airy Eigenfunctions}.  This alternative representation is the direct analogue of the representation in Theorem \ref{ET Theorem} for the periodic case, with a modified revival operator $\widetilde{\mathcal{R}_3}$ defined in terms of the eigenfunctions of the quasi-periodic problem directly. 
We state this representation, without proof. It can be obtained from algebraic manipulations of the expression \eqref{Airy QP Revival Formula}, or directly following the lines of the proof of Lemma~\ref{Revival Operator Lemma}. 

\begin{proposition}
Let $p,\,q,\,c,\,d$, $u_0(x)$ and $\theta_{\mathrm{r}}$ be as in the previous statement. Let $u(x,t)$ denote the solution of Airy's quasi-periodic problem \eqref{Airy QPP}, with $\theta=\theta_{\mathrm{r}}$.
The solution $u(x,t_{\mathrm{r}})$ at rational time $t_{\mathrm{r}} = 2\pi\frac{p}{q}$ admits the representation
\begin{equation}
\label{Airyaltrep}
u(x,t_{\mathrm{r}})=\frac{\sqrt{2\pi}}{d^2 q}\sum_{k=0}^{d^2q -1} \sum_{m=0}^{d^2 q-1}e^{-i(m+\frac c d)^3t}\phi_m \left(\frac{\pi k}{dq}\right)
\tilde{u}_0\left(x-\frac{\pi k}{2dq}\right).
\end{equation}
Here $\phi_m(x)$ are the eigenfunctions of the spatial operator given by \eqref{Airy Eigenfunctions} and $\tilde{u}_0(x)$ is the quasi-periodic extension of $u_0$, 
$$
\tilde{u}_0(x)=e^{2\pi i \frac{c}{d} m} u_0(x-2\pi m),\qquad 2\pi m\leq x<2\pi (m+1).
$$
\end{proposition}

In appendix~\ref{Numerical Examples}, we illustrate with several numerical examples the revival behaviour described by the results in this section. 

\begin{remark}
By an induction argument, the above results can generalise to higher order equations with a monic dispersion relation $P(k)=k^p$, $p\geq 4$.
\end{remark}


\section{The Linear Schr{\"o}dinger equation with Robin boundary conditions}

In this final section we consider the linear Schr\"{o}dinger equation \eqref{linS} posed on $(0,\pi)$, but now we impose the  Robin boundary conditions \eqref{rbc}. Namely, the problem we consider is
\begin{equation}
	\label{Robin Problem}
	\begin{aligned}
		& i u_{t} + u_{xx}=0, \quad u(x,0) = u_{0}(x)\in L^2(0,\pi),\\
		&b u(x_{0},t) = (1-b) \partial_{x}u(x_{0},t),\  x_{0} = 0, \ \pi,\quad b \in [0,1],
	\end{aligned}
\end{equation}
and we give the proof of Theorem \ref{Robin Revival Corollary}, whose results describes the behaviour of the solution of \eqref{Robin Problem} at rational times.

A routine calculation shows that the eigenvalues and the normalised eigenfunctions of the spatial operator are as follows. When $0<b<1$, there is one negative eigenvalue, which depends on the parameter $b$, given  by
\[
\lambda_{b} = - m_{b}^{2}<0 ,\quad m_{b} = \frac{b}{1-b},
\]
with associated normalised eigenfunction
\[
\phi_{b} (x) = A_{b} e^{m_{b}x}, \qquad A_{b} = \sqrt{\frac{2m_{b}}{e^{2\pi m_{b}}-1}}.
\]
The rest of the spectrum is the sequence of eigenvalues, independent of $b$, given by
$
\lambda_{j} = j^{2}>0$, $j\in\,\mathbb{N}
$
with associated normalised eigenfunctions,
\[ 
		 \phi_{j} (x) = \frac{1}{\sqrt{2\pi}} \left[e^{ijx} - \Lambda_{j} e^{-ijx}\right], \qquad \Lambda_{j} = \frac{b-(1-b)i j}{b+(1-b)ij}. 
\]
Note that the cases $b\rightarrow1$ and $b\rightarrow0$ correspond to Dirichlet and Neumann boundary conditions respectively. It is a routine calculation to verify that, by taking the even or odd extension, these can be treated as periodic problems posed on the double-length  interval $(0,2\pi)$. 

In order to simplify the presentation we set the following notation. For $f\in L^2(0,\pi)$, the \emph{even} and \emph{odd} extensions of $f$ to the segment $[0,2\pi]$ are denoted by
\begin{equation}
	\label{EvenOdd extension initial condition}
	 f^{\pm }(x)  = 
	\begin{cases}
		f(x),    & 0\leq x<\pi,  \\
		\pm f(2\pi-x), & \pi \leq x<2\pi,
	\end{cases} 
\end{equation}
and we  write the $2\pi$-periodic convolution of $f,\,g\in L^2(0,2\pi)$, as
	\begin{equation}
		\label{Periodic Convolution}
		f\ast g (x) = \frac{1}{\sqrt{2\pi}} \int_{0}^{2\pi} {f}^{*}(x-y) g^*(y) dy, \quad x\in(0,2\pi),
	\end{equation}
where the symbol $\ast$ on top of functions denotes the $2\pi$-periodic extension as in \eqref{Periodic Extension}.

Finally, as in the statement of Theorem \ref{Robin Revival Corollary}, we define
	\begin{equation}
		\label{f1f2}
		f_{1}(x) = \sqrt{\frac\pi 2}\frac{ m_{b}}{e^{2\pi m_{b}}-1} e^{m_{b}x},\qquad x\in(0,2\pi).
	\end{equation}

We first state a representation of the solution of \eqref{Robin Problem}  in terms of the solutions of five periodic problems for \eqref{linS}, each with an initial condition specified by an explicit transformation of $u_{0}$. Four of these initial conditions are  obtained as the $2\pi$-periodic convolution of an explicit exponential $2\pi$-periodic function with corresponding odd or even $2\pi$-periodic extensions of the initial data. 
\begin{proposition}\label{RCP}
	\label{Robin Connection Proposition}
	Let $u_{0}\in L^{2}(0,\pi)$, and consider the following solutions to the $2\pi$-periodic problem for equation \eqref{linS}:
\begin{itemize}
\item $n(x,t)$ denotes the solution corresponding to initial condition $n_{0}(x) = u_{0}^{+}(x)$ 
\item $h(x,t)$ denotes the solution corresponding to initial condition $h_{0}(x) = (f_{1}+f_{1}^\natural)\ast u_{0}^{+}(x)$ 
\item $v(x,t)$ denotes the solution corresponding to initial condition $v_{0}(x) = (f_{1}^\natural - f_{1})\ast u_{0}^{+}(x)$
\item $z(x,t)$ denotes the solution corresponding to initial condition $z_{0}(x) =(f_{1}-f_{1}^\natural)\ast u_{0}^{-}(x)$
\item $w(x,t)$ denotes the solution corresponding to initial condition $w_{0}(x) = (f_{1}+f_{1}^\natural) \ast u_{0}^{-}(x)$,
\end{itemize}
where $f_1(x)$ is  defined by \eqref{f1f2} and $\natural$ denotes reflection as given in \eqref{Reflected initial condition}. Then, at each $t\geq 0$ the solution $u(x,t)$ to the Robin problem \eqref{Robin Problem} is given by
	\begin{equation}
		\label{Robin Connection}
		u(x,t) =\langle u_{0},\phi_{b}\rangle_{L^{2}(0,\pi)} e^{im_{b}^{2}t} \phi_{b}(x) + n(x,t) - h(x,t) +v(x,t) + z(x,t) + w(x,t).
	\end{equation}
\end{proposition}
We omit the proof of this proposition, which  is  entirely analogous to the proof of Proposition~\ref{Non-Self-Adjoint Correspondence Theorem}. Various numerical examples 
which illustrate revival and non-revival for \eqref{Robin Problem} are given in Appendix~\ref{NumforA}. 

The proof of Theorem \ref{Robin Revival Corollary} is an immediate consequence of Proposition \ref{RCP}, which expresses the solution of this problem, in \eqref{Robin Revival}, as the sum of three terms:
\[
	\begin{aligned}
			u(x, t_{\mathrm{r}}) = &2\sqrt{\frac{2}{\pi}}
\langle u_{0},e^{\frac{b}{1-b}(\cdot)}\rangle_{L^{2}(0,\pi)} e^{i\frac{b^{2}}{(1-b)^{2}}t_{\mathrm{r}}} f_{1}(x) +
\mathcal{R}_2(p,q)\left[u_0^+(x)\right] \\ &+ \mathcal{R}_2(p,q)\left[2f_1*(u_0^--u_0^+))(x)\right], \quad x\in(0,\pi).
\end{aligned}
\]
The three components on the right hand side of the equation correspond to the following:
\begin{itemize}
\item[*]
The first term is a rank one perturbation and represents the contribution of the negative eigenvalue $\lambda_b$.
\item[*]
The second terms is the periodic revival of the (even extension of the) given initial condition.
\item[*]
The last term is the periodic revival of a continuous function.
\end{itemize}
As a consequence of this representation, we conclude that \eqref{Robin Problem} exhibits a {\em weaker form of revivals}. While the solution is not simply obtained as a linear combination of translated copies of the initial condition,  the second term in \eqref{Robin Revival} ensures that the {\em functional class of the initial condition is preserved at rational times}. In particular,  whenever $u_{0}$ has a finite number of jump discontinuities, then the same will be true for the solution at rational times, and the dichotomy between the solution behaviour at rational or irrational times is  present. We may say that the quantum particle that solves the linear Schr\"odinger equation with Robin boundary conditions still {\em knows the time}.

 \section*{Conclusions}\label{Conclusions}

The main goal of this work was to examine a variety of boundary conditions for the linear Schr\"odinger and  Airy equations,  and identify how the revival phenomenon depends on these boundary conditions. The starting point was the periodic case, for which it is known that the solution at rational times can be obtained as a finite linear combination of translated copies of the initial condition, and the dichotomy  between revival at rational times and fractalisation at irrational times is well established.  

We analysed pseudo-periodic conditions, which couple the two ends of the interval of definition, and Robin-type boundary conditions imposed separately at the two ends. We derived two main new results. One that establishes the constraints on the validity of the revival property for the third-order Airy equation. The other that describes a new, weaker form of revival for the case of Robin conditions. 

More specifically, we confirmed that in the second-order case of the linear Schr\"odinger equation,  every pseudo-periodic problem admits revival, by expressing its solution in terms of a purely periodic problem.  We then show, by virtue of this new expression,  that the revival property is more delicate for the third-order case of the Airy equation. In fact, it does not even hold in general for quasi-periodic boundary conditions. The rational/irrational time dichotomy, typical of  the revival phenomenon, holds in this case only for rational value of the quasi-periodicity parameter.

The particular case of Robin boundary conditions that we have chosen,   revealed a new weaker form of revival phenomenon, which is worth further investigation. In this case, while the rational/irrational time dichotomy  still holds, it is not true anymore that the solution at rational times is simply obtained by a finite linear combinations of copies of the initial profile.  It is worth highlighting that the validity of a form of revival in this case is due to the presence of one term in the solution representation that is due to a purely periodic problem. This new manifestation of  revival  complements the one recently reported in \cite{boulton2020new} for the case of periodic linear integro-differential equations. The latter displays a rational/irrational time dichotomy similar to the present one, but the representation of the solution is more involved.

Our analysis strongly support the conjecture that periodicity, and the number-theoretic properties of the purely exponential series that represent periodic solutions, are essential to any revival phenomenon.  Future work will aim to confirm this conjecture, by extending consideration to general  linear, constant coefficients boundary conditions for the both Sch\"rodinger and Airy equation. In the latter case, there exists boundary conditions for which the associate spatial operator does not admit a complete basis of eigenfunctions - an example of such conditions are the pseudo-Dirichlet conditions $u(0,t)=u(2\pi,t)=u_x(2\pi,t)=0$, see \cite{fokas2005transform, pelloni2005spectral}.  While preliminary numerical evidence suggests that at rational and irrational times the solution of this boundary value problem behaves fundamentally differently, the analysis for these types of boundary conditions requires a different approach. 

The equations we have considered are the linear part of important nonlinear equations of mathematical physics, the nonlinear Schr\"odinger and KdV equations respectively.  In work of Erdo\u{g}an, Tzirakis, Chousionis and Shakan, see \cite{erdogan2013talbot, chousionis2014fractal,erdougan2013global,erdougan2019fractal}, the dichotomy between the behaviour at rational and irrational times has been established rigorously for the periodic problem for these nonlinear equations. We expect that our result for the pseudo-periodic case would extend to the nonlinear case in an analogous manner. This would also provide theoretical foundation for recent results on the vortex filament equation with non-zero torsion \cite{de2020evolution}, a problem that can be represented in terms of the solution of a quasi-periodic problem for the Schr\"odinger equation.

 \section*{Acknowledgements}
 
We thank David Smith for his useful comments and suggestions on the contents of this paper. BP and LB are also grateful for the invitation to Yale-NUS College in January 2020 for a workshop funded by grant IG18-CW003, in which  discussions leading to part of this work began. GF is being supported by The Maxwell Institute Graduate School in Analysis and its Applications, a Centre for Doctoral Training funded by EPSRC (grant EP/L016508/01), the Scottish Funding Council, Heriot-Watt University and the University of Edinburgh.

 \addcontentsline{toc}{section}{References} 
 \printbibliography

\appendix

 \section{Numerical examples for Airy's equation}\label{Numerical Examples}
 In this first appendix, we display the numerical solutions of the quasi-periodic problem for Airy's equation  \eqref{Airy  QPP}. We  illustrate  the phenomenon of revivals and fractalisation for two choices of the quasi-periodicity parameter, one rational and one irrational. The initial condition is the piecewise constant function: $u_{0}(x)=0$ in $(0,\pi)$ and $u_{0}(x)=1$ in $(\pi,2\pi)$.
 
 In figures~\ref{Airy Rt 1/4} and \ref{Airy Vt 1/4}, we plot the profile of the solution in space variable, with quasi-periodic boundary conditions determined by $\beta=e^{2\pi i \theta}$ with $\theta\in\mathbb Q$.  In the first figure, the time is set to be rational. The re-appearance of the initial jump discontinuity is clearly seen. In the second figure,  the time is irrational. The solution has both its real and imaginary parts continuous. Indeed, the discontinuity has been smoothed out. This is consistent with Oskolkov's results \cite{oskolkov1992class}, stating that for the periodic problem for linear Schr\"{o}dinger and Airy equations, at irrational times, the solution is a continuous functions of $x$  provided the initial condition is of bounded variation.
 
  In figures~\ref{Airy Rt square} and \ref{Airy Vt square}, we plot the solution with quasi-periodic boundary conditions determined by $\beta=e^{2\pi i \theta}$ with $\theta\notin\mathbb Q$.   In this case, no discontinuities appear in the solution at any time, either rational or irrational.  This is consistent with the representation \eqref{Airy Correspondence} of the solution of this problem in terms of the solution of a periodic problem for the Schr\"odinger equation at an {\em irrational} time.
  
 \begin{figure}[H]
 \centering\includegraphics[width=0.71\textwidth]{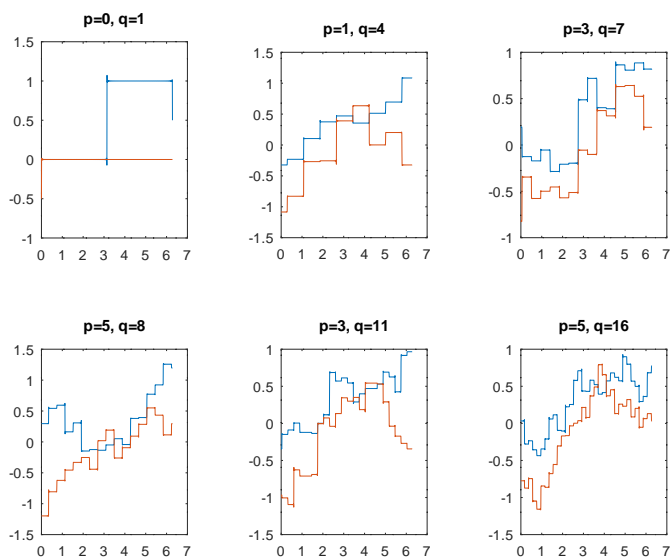}
 \caption{Real (blue) and imaginary (red) parts of the solution of Airy's problem  \eqref{Airy QPP}  with $\theta=1/4$ at rational times $t = 2\pi p/q$.}
	\label{Airy Rt 1/4}
 \end{figure}

  \begin{figure}[H]
 \centering\includegraphics[width=0.71\textwidth]{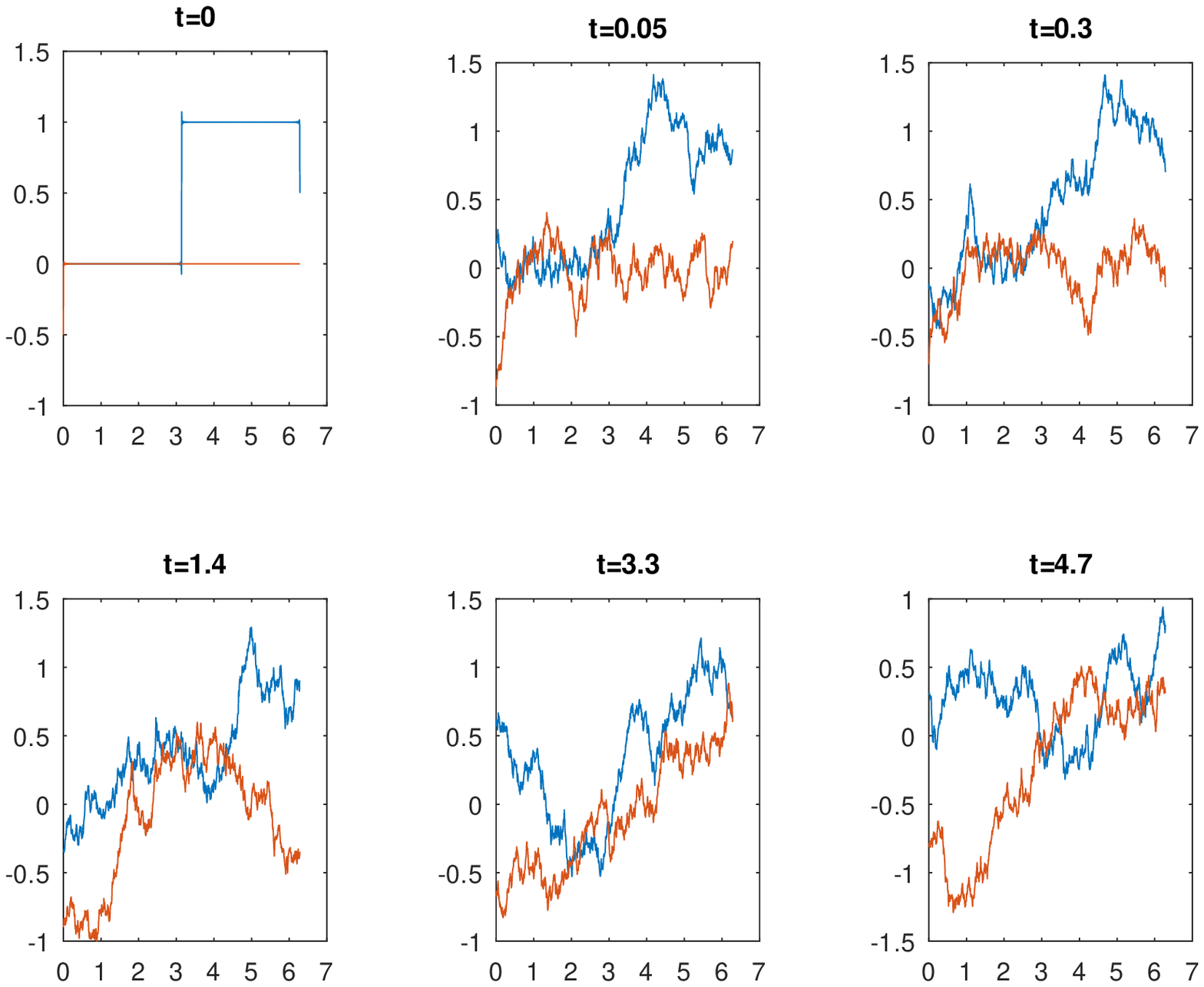}
	\caption{Real (blue) and imaginary (red) parts of the solution of Airy's problem  \eqref{Airy QPP}  with $\theta=1/4$ at generic times.}
	\label{Airy Vt 1/4}
 \end{figure} 

  \begin{figure}[H]
 \centering\includegraphics[width=0.71\textwidth]{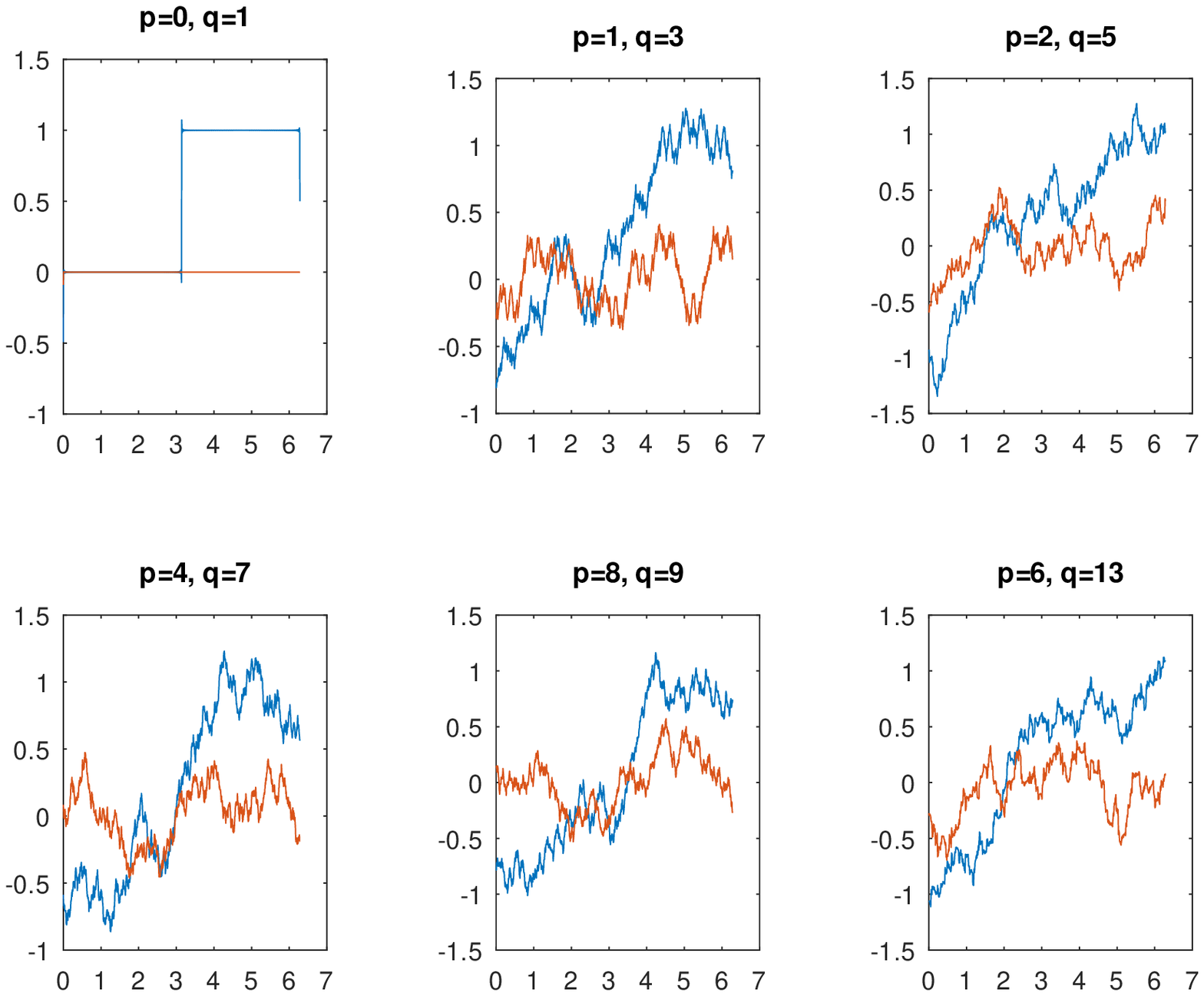}
 	\caption{Real (blue) and imaginary (red) parts of the solution of Airy's problem  \eqref{Airy QPP} with $\theta=\sqrt{2}/3$ at rational times $t = 2\pi p/q$.}
 	\label{Airy Rt square}
 \end{figure}

 \begin{figure}[H]
 \centering\includegraphics[width=0.71\textwidth]{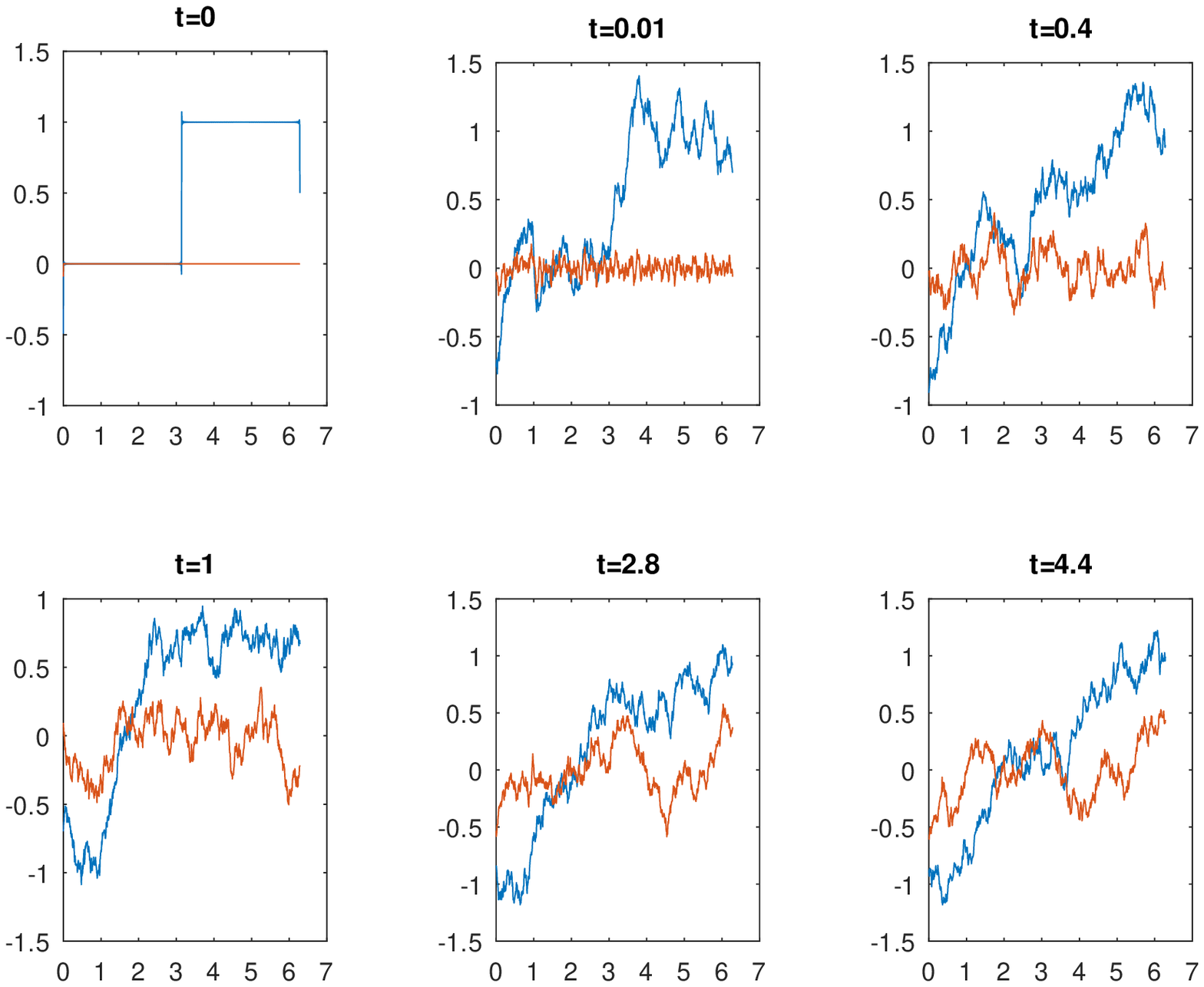}
 	\caption{Real (blue) and imaginary (red) parts of the solution of Airy's problem  \eqref{Airy QPP}  with $\theta=\sqrt{2}/3$ at generic times.}
 	\label{Airy Vt square}
 \end{figure}
 
\section{Numerical Examples for the Robin linear Schr\"{o}dinger problem}
\label{NumforA}
These final numerical experiments correspond to the equation \eqref{Robin Problem}. We take as initial condition the piecewise constant function: $u_{0}(x) = 0$ in $(0,\pi/2)$ and $u_{0}(x) = 1$ in $(\pi/2,\pi)$.
Picking different values of the parameter $b\in [0,1]$, we plot the real and imaginary part of the solution $u(x,t)$, in space, at generic and rational times.

At rational times, in figures~\eqref{robin fig 1_1} and \eqref{robin fig 2_1}, we notice that the solution evolves to, not exactly, only translations and/or reflections of the initial profile. However, the revival of the discontinuities is preserved, as predicted by Theorem~\ref{Robin Revival Corollary}. On the other hand, see figures~\ref{robin fig 1_2} and \ref{robin fig 2_2}, at generic times the solution profile is clear of discontinuities.
 \begin{figure}[H]
 \centering\includegraphics[width=0.71\textwidth]{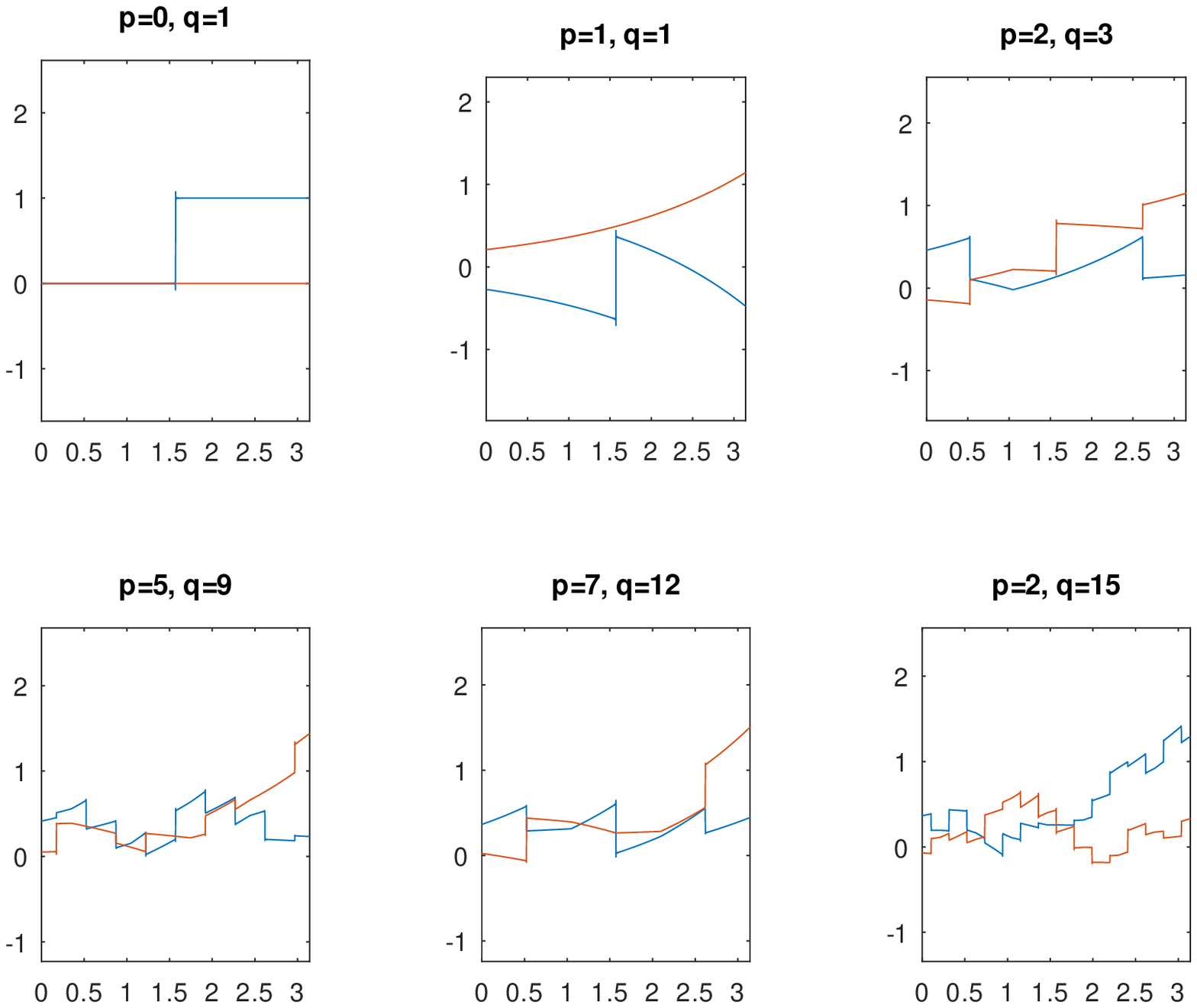}
	\caption{Real (blue) and imaginary (red) part of the solution of Robin's problem \eqref{Robin Problem} with $b=0.35$ at rational times $t=2\pi p /q$.}
	\label{robin fig 1_1}
\end{figure}

 \begin{figure}[H]
 \centering\includegraphics[width=0.71\textwidth]{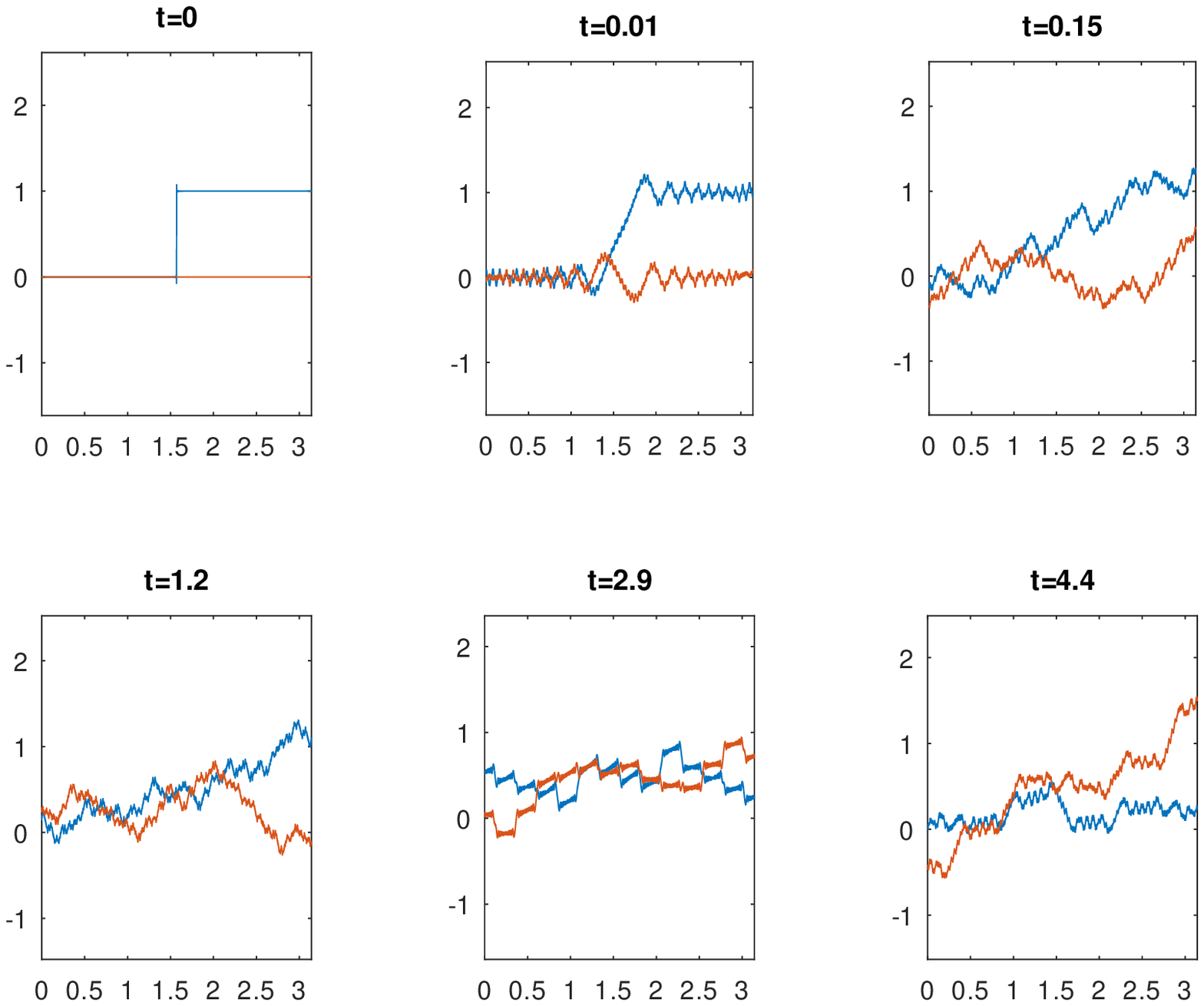}
	\caption{Real (blue) and imaginary (red) parts of the solution of Robin's problem \eqref{Robin Problem} with $b=0.35$ at generic times.}
	\label{robin fig 1_2}
\end{figure}

 \begin{figure}[H]
 \centering\includegraphics[width=0.71\textwidth]{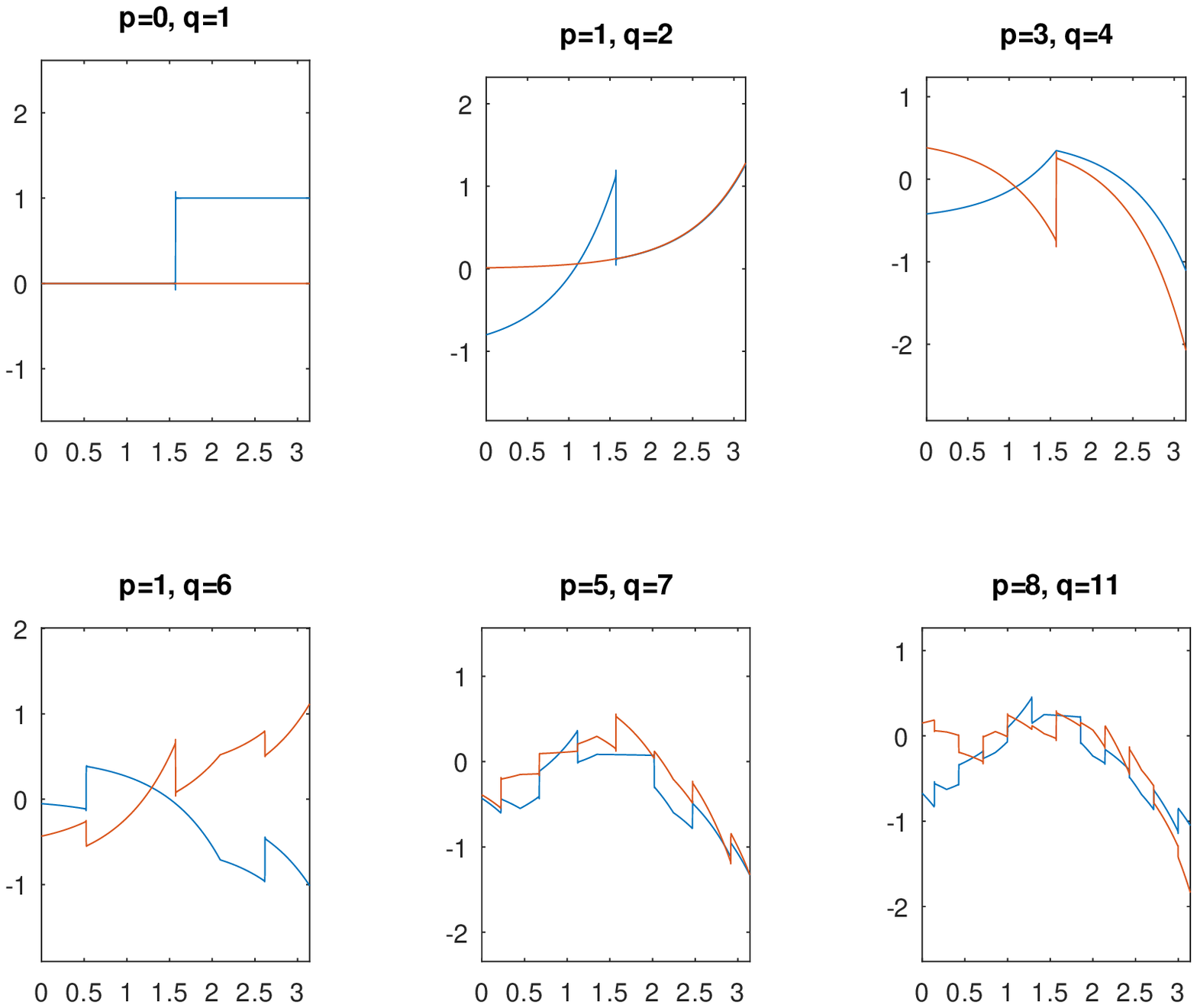}
	\caption{Real (blue) and imaginary (red) parts of the solution of Robin's problem \eqref{Robin Problem} with $b = 0.6$ at rational times $t=2\pi p /q$.}
	\label{robin fig 2_1}
\end{figure}

 \begin{figure}[H]
 \centering\includegraphics[width=0.71\textwidth]{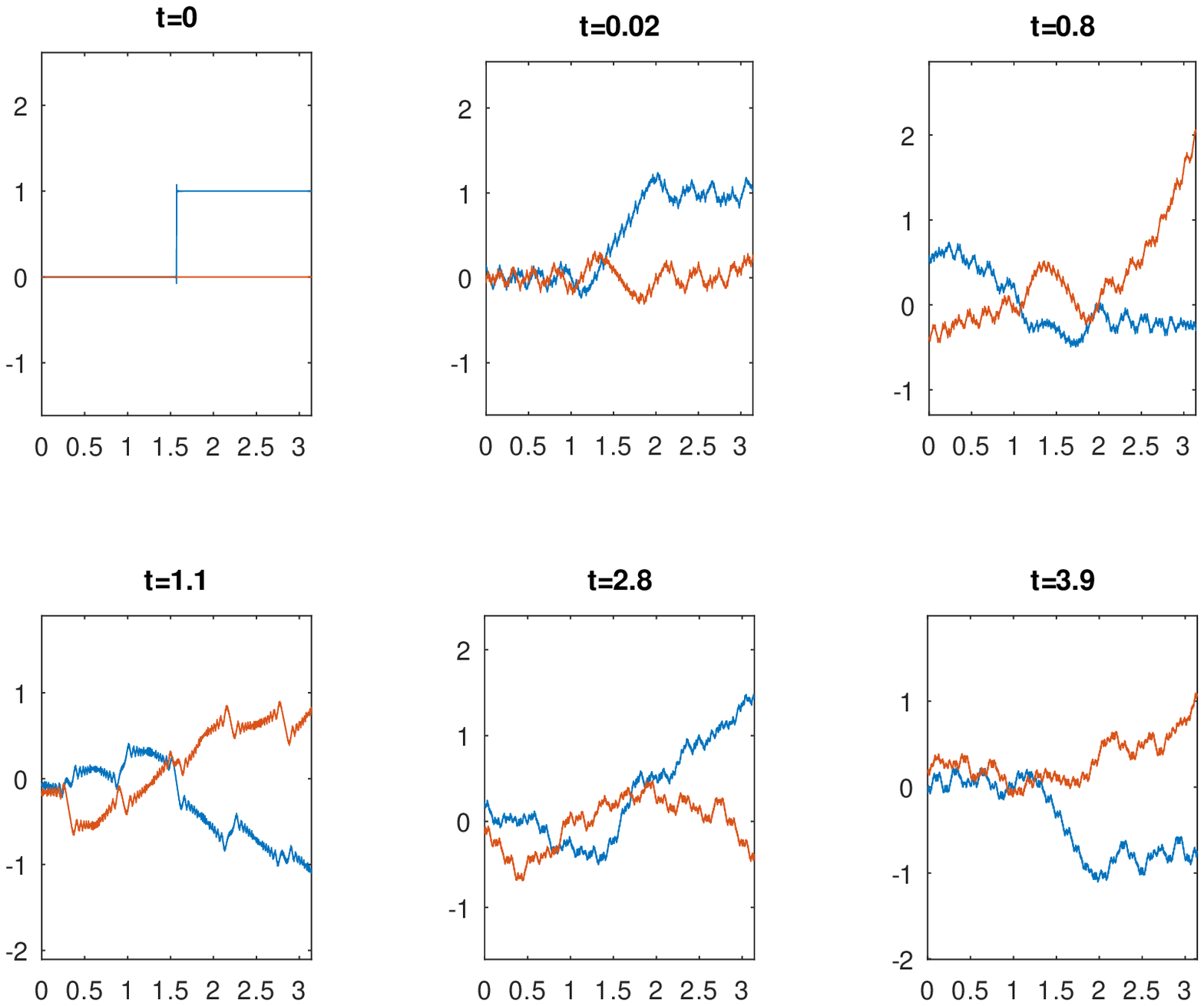}
	\caption{Real (blue) and imaginary (red) parts of the solution of Robin's problem \eqref{Robin Problem} with $b = 0.6$ at generic times.}
	\label{robin fig 2_2}
\end{figure}

\end{document}